\documentclass[final]{siamltex}

\usepackage{amssymb,latexsym,amsmath}
\usepackage{mathrsfs}
\usepackage{epsfig}
\usepackage{caption}

\newtheorem{example}[theorem]{Example}

\def\qed{$\endproof$}
\fussy

\begin{document}

\allowdisplaybreaks

\title{Optimization and Convergence of Observation Channels in
  Stochastic Control} \author{Serdar Y\"uksel and Tam\'as Linder}
\maketitle
{\renewcommand{\thefootnote}{}

\footnotetext{The material of this paper was presented in part at the
  2010 Information Theory and Applications Workshop, University of
  California, San Diego, Feb.\ 2010.}

\footnotetext{The authors are with the Department of Mathematics and
    Statistics, Queen's University, Kingston, Ontario, Canada, K7L
    3N6.  Email: (yuksel,linder)@mast.queensu.ca. This research was
    partially supported by the Natural Sciences and Engineering
    Research Council of Canada (NSERC). }}

\begin{abstract}
This paper studies the optimization of observation channels
(stochastic kernels) in partially observed stochastic control
problems. In particular, existence and  continuity
properties are investigated mostly (but not exclusively) concentrating
on the single-stage case. Continuity properties of the optimal cost
in channels are explored under total variation, setwise convergence,
and weak convergence. Sufficient conditions for compactness of a class
of channels under total variation and setwise convergence are
presented and applications to quantization are explored.  
\end{abstract}


\pagestyle{myheadings}
\thispagestyle{plain}

\section{Introduction}

In stochastic control, one is often concerned with the following
problem: Given a dynamical system, an observation channel (stochastic
kernel), a cost function, and an action set, when does there exist an
optimal policy, and what is an optimal control policy? The theory for
such problems is advanced, and practically significant, spanning a
wide variety of applications in engineering, economics, and natural sciences.

In this paper, we are interested in a dual problem with the following
questions to be explored: Given a dynamical system, a cost function,
an action set, and a set of observation channels, does there exist an
optimal observation channel?  What is the right convergence notion for
continuity in such observation channels for optimization purposes? The
answers to these questions may provide useful tools for characterizing
an optimal observation channel subject to constraints.

We start with the probabilistic setup of the problem. Let $\mathbb{X}
\subset \mathbb{R}^n$, be a Borel set in which elements of a
controlled Markov process $\{X_t,\, t \in \mathbb{Z}_+\}$ live.  Here
and throughout the paper $\mathbb{Z}_+$ denotes the set of nonnegative
integers and $\mathbb{N}$ denotes the set of positive integers.  Let
$\mathbb{Y} \subset \mathbb{R}^m$ be a Borel set, and let an
observation channel $Q$ be defined as a stochastic kernel (regular
conditional probability) from  $\mathbb{X}$ to $\mathbb{Y}$, such that
$Q(\,\cdot\,|x)$ is a probability measure on the (Borel)
$\sigma$-algebra ${\cal B}(\mathbb{Y})$ on $\mathbb{Y}$ for every $x
\in \mathbb{X}$, and $Q(A|\,\cdot\,): \mathbb{X}\to [0,1]$ is a Borel
measurable function for every $A \in {\cal B}(\mathbb{Y})$. Let a
decision maker (DM) be located at the output an observation channel
$Q$, with inputs $X_t$ and outputs $Y_t$.  Let $\mathbb{U}$ be a Borel
subset of some Euclidean space. An {\em admissible policy} $\Pi$ is a
sequence of control functions $\{\gamma_t,\, t\in \mathbb{Z}_+\}$ such
that $\gamma_t$ is measurable with respect to the $\sigma$-algebra
generated by the information variables
\[
I_t=\{Y_{[0,t]},U_{[0,t-1]}\}, \quad t \in \mathbb{N}, \quad
  \quad I_0=\{Y_0\}.
\]
where
\begin{equation}
\label{eq_control}
U_t=\gamma_t(I_t),\quad t\in \mathbb{Z}_+
\end{equation}
are the $\mathbb{U}$-valued control
actions and  we used the notation
\[
Y_{[0,t]} = \{Y_s,\, 0 \leq s \leq t \}, \quad U_{[0,t-1]} =
  \{U_s, \, 0 \leq s \leq t-1 \}.
\]
The joint distribution of the state, control, and observation
processes is determined by (\ref{eq_control}) and the following
relationships:
\[
  \Pr\bigl( (X_0,Y_0)\in B \bigr)
=  \int_B P(dx_0)Q(dy_0|x_0), \quad B\in \mathcal{B}(\mathbb{X}\times
\mathbb{Y}),
\]
where $P$ is the (prior) distribution of the initial state $X_0$, and
\begin{eqnarray*}
\label{eq_evol}
\lefteqn{  \Pr\biggl( (X_t,Y_t)\in B \, \bigg|\, X_{[0,t-1]}=x_{[0,t-1]},
  Y_{[0,t-1]}=y_{[0,t-1]}, U_{[0,t-1]}=u_{[0,t-1]} \biggr)
} \hspace{3cm} \\
\mbox{} \quad  \hspace{-4cm} &=& \int_B P(dx_t|x_{t-1}, u_{t-1})Q(dy_t|x_t), \quad  B\in \mathcal{B}(\mathbb{X}\times
\mathbb{Y}), \quad t\in \mathbb{N},
\end{eqnarray*}
where $P(\cdot|x,u)$ is a stochastic kernel from $\mathbb{X}\times
\mathbb{U}$ to $\mathbb{X}$.

One way of presenting the problem in a familiar setting is the
following: Consider a dynamical system described by the
discrete-time equations
\begin{eqnarray*}
X_{t+1} &=& f(X_t,U_t, W_t),\\
Y_t&=& g(X_t,V_t)
\end{eqnarray*}
for some measurable functions $f, g$, with
$\{W_t\}$ being independent and identically distributed (i.i.d) system
noise process and $\{V_t\}$ an i.i.d. disturbance process, which are
independent of $X_0$ and each other. Here, the second equation
represents the communication channel $Q$, as it describes the relation
between the state and observation variables.

With the above setup, let the objective of the decision maker be the
minimization of the cost
\begin{eqnarray}\label{Cost}
J(P,Q,\Pi)=E_{P}^{Q,\Pi}\bigg[\sum_{t=0}^{T-1} c(X_t,U_t)\bigg],
\end{eqnarray}
over the set of all admissible policies $\Pi$, where
$c:\mathbb{X}\times \mathbb{U}\to \mathbb{R}$ is a Borel
measurable cost function and $E_{P}^{Q,\Pi}$ denotes the expectation
with initial state probability measure given by $P$ under policy $\Pi$
and given channel $Q$. We adapt the convention that  random variables
are denoted by capital letters  and lowercase
letters denote their realizations. Also, given a probability measure
$\mu$ the notation $Z\sim \mu$ means that $Z$ is a random variable
with distribution~$\mu$. Finally, let $\mathscr{P}$ be the set of all admissible policies $\Pi$ described above.

We are interested in the following problems:

\noindent\textsc{Problem P1. Continuity on the space of channels
  (stochastic kernels)} \ Suppose $\{Q_n, n \in \mathbb{N}\}$ is a
sequence of communication
channels converging in some sense to a channel $Q$. When does \[
Q_n \to Q
\]
imply
\[ \inf_{\Pi \in \mathscr{P}} J(P,Q_n,\Pi)  \to \inf_{\Pi \in
  \mathscr{P}} J(P,Q,\Pi)?
\]

\smallskip

\noindent\textsc{Problem P2: Existence of optimal channels} \
Let ${\cal Q}$ be a set of communication channels. When do there exist
minimizing and maximizing channels for the problems
\[
\inf_{Q \in {\cal Q}} \inf_{\Pi  \in \mathscr{P}}
E^{Q,\Pi}_{P}\bigg[\sum_{t=0}^{T-1} c(X_t,U_t)\bigg]
\]
and
\[\sup_{Q \in {\cal Q}} \inf_{\Pi  \in \mathscr{P}}
E^{Q,\Pi}_{P}\bigg[\sum_{t=0}^{T-1} c(X_t,U_t)\bigg].
\]
If solutions to these problems exist, are they unique?

Problems \textsc{P1} and \textsc{P2} are challenging even in the single-stage
($T=1$) setup and in most of the paper we consider this
case. Admittedly, the multi-stage case is more important and we
briefly consider this case is Section~\ref{secmulti} at the end of the
paper.  Future work is needed to fully address this technically more
complex case.

The answers to problems  \textsc{P1} and \textsc{P2} may help solve problems in
application areas
such as:
\begin{itemize}
\item For a partially observed stochastic control problem, sometimes
  we have control over the observation channels by
  encoding/quantization. When does there exist an optimal quantizer
  for such a setup? (Optimal quantization)
\item Given an uncertainty set for the observation channels, can one
  identify a worst element/best element? (Robust control)

\item When estimating channels from empirical observations, under
quite general assumptions estimations converge to the actual
distribution, in some sense. For example, if an observation channel
has the form $Y_t= X_t + V_t$, where the independent noise $V_t$ has a density,
nonparametric density
estimation methods lead to convergence in total variation, whereas
for the general case, the empirical measures converge weakly with
probability one \cite{DevroyeGyorfi}, \cite{Dud02}. Do these modes of
convergence  imply
that we could design the optimal control policies based on empirical
estimates, and does the optimal cost converge to the correct limit as
the number of measurements grows? (Consistency of empirical
controllers)
\end{itemize}

In the following, we will address problems  \textsc{P1} and
\textsc{P2} and introduce
conditions under which we can provide affirmative/conclusive answers.

\subsection{Relevant literature}
The problems stated are related to three main areas of research:
Robust control, optimal quantizer design and design of experiments. 

References  \cite{Charalambous,Charalambous15,Charalambous2} have
considered  both
optimal control and
estimation and the  related problem of optimal control design when
the channel is unknown. In particular, \cite{Charalambous2}
studies the existence of optimal continuous estimation policies and
worst-case channels under a relative entropy constraint characterizing
the uncertainty in the system. In \cite{Charalambous15}, the total
variation norm is considered as the measure of the uncertainty, and
the inf-sup policy is determined (thus, the setup considered as a
min-max problem for the generation of optimal control
policies). Similarly, there are connections with robust detection,
such as those studied by Huber \cite{Huber} and Poor \cite{Poor}, when
the source distribution to be detected belongs to some set.

A related area is on the theory of optimal quantization:  References
\cite{AbayaWise}, \cite{GrayDavisson} are related as these papers
study the effects of uncertainties in the input distribution and
consider robustness in the quantizer design. References \cite{Linder}
and \cite{Pollard} study the consistency of optimal quantizers based
on empirical data for an unknown source. In the context of
decentralized detection, \cite{Tsitsiklis} studied certain topological
properties and the existence of optimal quantizers.  We will regard
the quantizers as a particular class of channels, and look for such
optimal channels. One by-product of our analysis will be a new approach
to obtain conditions for the existence of optimal quantizers for a
given class of cost functions under mild conditions.  We also note
that, regarding connections with information theory, some discussions
on the topology of information channels are presented in
\cite{MurakiOhya}. Recently, \cite{WuVerdu} considered continuity and
other functional properties of minimum mean square estimation problems
under Gaussian channels.

As mentioned earlier, in most of the paper we consider the single-stage case. We will
also briefly consider the technically more complex multi-stage case in Section~\ref{secmulti} where
further conditions on the controlled Markov chain must be imposed. The
full development of this general setup is the subject of future work.

The rest of the paper is organized as follows. In the next section, we
introduce three relevant topologies on the space of communication
channels. The continuity problem is considered in
Section~\ref{sectionCont}.
We study the
problem of existence of optimal channels in
Section~\ref{sectionExistence}, followed by applications on
quantization in
Section~\ref{secapp}. Section~\ref{secmulti}
gives an outlook to the multi-stage setup. The paper ends with the concluding remarks and
discussions in Section~\ref{sectionConclusion}.

\section{Some topologies on the space of communication channels}

One question that we wish address is the choice of an appropriate
notion of convergence for a sequence of observation channels.  Toward
this end, we first review three notions of convergence for probability
measures.

Let $\mathcal{P}(\mathbb{\mathbb{R}^N})$ denote the family of
all probability measure on $(\mathbb{X},\mathcal{B}(\mathbb{R}^N))$
for some $N\in \mathbb{N}$.
Let $\{\mu_n,\, n\in \mathbb{N}\}$ be a sequence in
$\mathcal{P}(\mathbb{R}^N)$. Recall that
$\{\mu_n\}$ is said to  converge
  to $\mu\in \mathcal{P}(\mathbb{R}^N)$ \emph{weakly} if
\[
 \int_{\mathbb{R}^N} c(x) \mu_n(dx)  \to \int_{\mathbb{R}^N}c(x) \mu(dx)
\]
for every continuous and bounded $c: \mathbb{R}^N \to \mathbb{R}$.
On the other hand, $\{\mu_n\}$ is said to  converge
  to $\mu\in \mathcal{P}(\mathbb{R}^N)$  \emph{setwise} if
\[
 \int_{\mathbb{R}^N} c(x) \mu_n(dx)  \to \int_{\mathbb{R}^N} c(x) \mu(dx)
\]
for every measurable and bounded $c: \mathbb{R}^N \to
\mathbb{R}$. Setwise convergence  can also be defined through
pointwise convergence on Borel subsets of $\mathbb{R}^N$ (see, e.g.,
\cite{HernandezLermaLasserre}), that is
\[
 \mu_n(A)   \to \mu(A),  \quad \text{for
   all $A \in{\cal B}(\mathbb{R}^N)$}
\]
since the space of simple functions is dense in the space of bounded
and measurable functions under the supremum norm.

For two probability measures $\mu,\nu \in
\mathcal{P}(\mathbb{R}^N)$, the \emph{total variation} metric
is given by
\begin{eqnarray}
\|\mu-\nu\|_{TV}&:= & 2 \sup_{B \in {\cal B}(\mathbb{R}^N)}
|\mu(B)-\nu(B)| \nonumber \\
 &=&  \sup_{f: \, \|f\|_{\infty} \leq 1} \bigg| \int f(x)\mu(dx) -
\int f(x)\nu(dx) \bigg|, \label{TValternative}
\end{eqnarray}
where the supremum is over all measurable real $f$ such that
$\|f\|_{\infty} = \sup_{x \in \mathbb{R}^N} |f(x)|\le 1$.
A sequence  $\{\mu_n\}$ is said to  converge
  to $\mu\in \mathcal{P}(\mathbb{R}^N)$ in total variation if
$\| \mu_n - \mu   \|_{TV}  \to 0.$

Setwise convergence is equivalent to pointwise convergence on Borel
sets whereas convergence in total variation requires uniform
convergence on Borel sets. Thus convergence in total variation implies
setwise convergence, which in turn implies weak convergence. It
follows that the induced topologies are of decreasing order of
strength,  with
the topology induced by convergence in total variation being the
strongest and the topology induced by weak convergence being the
weakest, with the topology induced by setwise convergence is in
between these two.  The topologies corresponding to convergence in
total variation and weak convergence are metrizable (the natural
metric for total variation convergence is
$d(\mu,\nu)=\|\nu-\nu\|_{TV}$; the usual choice for weak convergence 
is the Prohorov metric \cite{Billingsley}). The topology induced by
setwise convergence is not first countable, so it is not
metrizable (see, e.g., \cite[Prop.\ 2.2.1]{Ghosh}).

\subsection{Convergence of information (observation) channels}

Here $\mathbb{X}=\mathbb{R}^n$ and $\mathbb{Y}=\mathbb{R}^m$, and
$\mathcal{Q}$ denotes the set of all observation channels (stochastic
kernels) with input space $\mathbb{X}$ and output space
$\mathbb{Y}$. For $P\in  \mathcal{P}(\mathbb{X})$ and $Q\in
\mathcal{Q}$ we let $PQ$ denote the joint distribution induced on
$(\mathbb{X}\times \mathbb{Y}, \mathcal{B}(\mathbb{X}\times
\mathbb{Y}))$ by channel $Q$ with input distribution $P$:
\[
  PQ(A) = \int_{A} Q(dy|x)P(dx), \quad A\in  \mathcal{B}(\mathbb{X}\times
\mathbb{Y}).
\]

\begin{definition}[Convergence of Channels] \label{ch_conv_def}
  \begin{romannum}
\item A sequence of channels $\{Q_n\}$ converges to a channel $Q$
\emph{weakly at input $P$} if $PQ_n\to PQ$ weakly.

\item A sequence of channels $\{Q_n\}$ converges to a channel $Q$
\emph{setwise at input $P$} if $PQ_n\to PQ$ setwise, i.e., if
$PQ_n(A)\to PQ(A)$
for all Borel sets  $A \subset \mathbb{X}\times
\mathbb{Y}$.
\item
A sequence of channels $\{Q_n\}$ converges to a channel $Q$ in
\emph{total variation at input $P$} if $PQ_n\to PQ$ in total variation,
i.e., if $\| PQ_n - PQ \|_{TV}  \to 0$.
\end{romannum}
\end{definition}

If we introduce the equivalence relation $Q\overset{P}\equiv Q'$ if and only if
$PQ=PQ'$, $ Q,Q'\in \mathcal{Q}$, then the convergence notions in
Definition~\ref{ch_conv_def} only induce the corresponding topologies
(resp.\  metrics) on the resulting equivalence classes in $\mathcal{Q}$,
instead of $\mathcal{Q}$. Since in most of the development the
input distribution $P$ is fixed, there should be no confusion when
(somewhat incorrectly) we talk about the induced   topologies (resp.\
metrics)  on $\mathcal{Q}$.

The preceding definition involved the input distribution $P$. The
next lemma gives sufficient conditions which may be easier to
verify. The proof is given in the Appendix.

\pagebreak[2]

\begin{lemma}\mbox{} \label{UniversalInputCompactness}
  \begin{romannum}
  \item If  $\{Q_n(\, \cdot\, |x)\}$ converges to $Q(\, \cdot\, |x)$
    weakly for $P$-a.e.\ $x$, then $PQ_n\to PQ$ weakly.

  \item If  $\{Q_n(\, \cdot\, |x)\}$ converges to $Q(\, \cdot\, |x)$
   setwise  for $P$-a.e.\ $x$, then $PQ_n\to PQ$ setwise.

  \item  If  $\{Q_n(\, \cdot\, |x)\}$ converges to $Q(\, \cdot\, |x)$
   in total variation  for $P$-a.e.\ $x$, then $PQ_n\to PQ$ in total variation.
  \end{romannum}
\end{lemma}

The conditions in Lemma \ref{UniversalInputCompactness} are almost
universal in the choice of input probability measures; that is, the
convergence characterizations will be independent of the input
distributions if each of the conditions is replaced with convergence
of $\{Q_n(\, \cdot\, |x)\}$ to $Q(\, \cdot\, |x)$ for all $x \in
\mathbb{X}$. This is particularly useful when the input distribution
is unknown, or when the input distributions may change. The latter can
occur in multi-stage stochastic control problems.

\pagebreak[4]

\begin{example} \rm \begin{romannum} 
\item Consider the case where the observation channel has the form
  $Y_t=X_t+V_t$, where $\{V_t\}$ is an i.i.d.\ noise (disturbance)
  process. Suppose $V_t\sim f_{\theta_0}$ for some $\theta_0 \in \Theta$,
  where $\Theta \subset \mathbb{R}^d$ is a parameter set and
  $\{f_{\theta}: \theta\in \Theta\}$ is a parametric family of
  $n$-dimensional densities such that $f_{\theta_n}(v)\to
  f_{\theta_0}(v)$ for all $v\in \mathbb{R}^n$ and any sequence of
  parameters $\theta_n$ such that $\theta_n \to \theta_0$.  Then by
  Scheff\'e's theorem $f_{\theta_n}$ converges to $f_{\theta_0}$ in the
  $L_1$ sense, and consequently, the sequence of corresponding
  additive channels $Q_n(\,\cdot\,|x)$, defined by
\[
  Q_n(A|x)= \int_{A} f_{\theta_n}(z-x)\, dz, \quad  A\in {\cal B}(\mathbb{R}^n)
\]
converges to the channel
  $Q(\,\cdot\,|x)$ (corresponding to $f_{\theta}$) in total variation for
  all $x$.
\item Consider again the observation channel $Y_t=X_t+V_t$, but assume
  this time that we only know that $V_t$ has a density $f$ (which is
  unknown to us). If we have access to independent observations
  $V_1,\ldots,V_n$ from the noise process, then we can use any of the
  consistent nonparametric methods, e.g., \cite{DevroyeGyorfi}, to
  obtain an estimate $f_n$ which converges (with probability one) to
  $f$ in the $L_1$ sense as $n\to \infty$.  More explicitly, letting
  $(\Omega,\mathcal{A},\mathbb{P})$ be the probability space on which the
  independent observations $\{V_i\}$ are defined, for any $\omega\in
  \Omega$, the estimate $f_n= f_{n,\omega}$ is a pdf on
  $\mathbb{R}^n$, and there 
  exists $A\in \mathcal{A}$ with $\mathbb{P}(A)=1$ such that
  $\int|f_{n,\omega}(z)- f_n(z)|\, dz\to 0$ as $n\to \infty$ for all
  $\omega\in A$.  The estimated channel
  $Q_n(\,\cdot\,|x)=Q_{n,\omega}(\,\cdot\,|x)$ corresponding to
  $f_{n,\omega}$ converges to the true channel $Q(\,\cdot\,|x)$ in
  total variation for all $x$ with probability one. More explicitly, 
  for any $\omega \in A$, $Q_{n,\omega}(\,\cdot\,|x)$ converges to
  $Q(\,\cdot\,|x)$ in total variation as $n\to \infty$ for all $x$.

\item Now suppose that the observation channel $Q$ is such that
  $Q(\,\cdot\,|x)$ admits a conditional density $f(y|x)$ for all $x\in
  \mathbb{R}^n$. Given observations $(X_1,Y_n), \ldots, (X_n,Y_n)$
  drawn independently from the distribution $PQ$, there exists a
  sequence of nonparametric conditional density estimates $f_n(y|x)$
  such that
\[
  \int \biggl(\int |f_n(y|x)-f(y|x)|\, dy\biggr) P(dx) \to 0
\]
with probability one \cite{GyKo07}. This immediately implies that the channels
$Q_n$ corresponding to these estimates converge to $Q$ in total
variation at input $P$.

\item Finally, assume again the additive model $Y_t=X_t+V_t$, where
  now we do not have any information about the distribution $\mu$ of
  $V_t$. In this case there are no methods to consistently estimate
  $\mu$ in total variation from independent samples $V_1,\ldots,V_n$
  \cite{DeGy90}. However, the empirical distribution $\mu_n$ of the
  samples converges weakly to $\mu$ with probability one
  \cite{Dud02}. The corresponding estimated observation channels
  $Q_n(\,\cdot\,|x)$ converge weakly to the true channel
  $Q(\,\cdot\,|x)$ for all $x$ with probability one.

\end{romannum}
\end{example}

\subsection{Classes of assumptions}

Throughout the paper the following classes of assumptions will be
adopted for the cost function $c$ and the (Borel) set
$\mathbb{U}\subset \mathbb{R}^k$ in different contexts:

\pagebreak[2]

\noindent\textsc{Assumptions.}

\nopagebreak

\begin{itemize}
\item[\textsc{A1}.] The  function  $c: \mathbb{X} \times \mathbb{U} \to
  \mathbb{R}$ is  non-negative, bounded, and continuous on $\mathbb{X}
  \times \mathbb{U}$.

\item[\textsc{A2}.] The function function $c: \mathbb{X} \times
  \mathbb{U} \to \mathbb{R}$ is non-negative, measurable, and bounded.

\item[\textsc{A3}.] The  function  $c: \mathbb{X} \times \mathbb{U} \to
  \mathbb{R}$  is  non-negative, measurable, bounded,
  and continuous on $\mathbb{U}$
  for every $x \in \mathbb{X}$.

\item[\textsc{A4}.] $\mathbb{U}$ is a compact set.

\item[\textsc{A5}.] $\mathbb{U}$ is a convex set.
\end{itemize}

\section{Problem {\bf P1}: Continuity of the optimal cost in
  channels}\label{sectionCont}

In this section, we consider continuity properties under total
variation, setwise convergence and weak convergence. We consider the
single-stage case, and thus investigate the continuity of the functional
\begin{eqnarray*}
J(P,Q)&= & \inf_{\Pi} E^{Q,\Pi}_{P}\big[ c(X_0,U_0)\big]  \\
   &=& \inf_{\gamma\in \mathcal{G}} \int_{\mathbb{X}\times \mathbb{Y}} c(x,\gamma(y)) Q(dy|x)P(dx)
\end{eqnarray*}
in the channel $Q$, where $\mathcal G$ is the collection of all Borel
measurable functions mapping $\mathbb{Y}$ into $\mathbb{U}$.  Note
that by our previous notation, $\Pi=\gamma$ is an admissible
first-stage control policy. As before, in this section $\mathcal{Q}$
denotes the set of all channels with input space $\mathbb{X}$ and
output space $\mathbb{Y}$.

Total variation is a stringent notion for convergence. For example a
sequence of discrete probability measures never converges in total
variation to a probability measure which admits a density function
with respect to the Lebesgue measure.  On the other hand, setwise
convergence induces a topology on the space of probability measures
and channels which is not easy to work with.  This is mainly due to
the property that the space under this convergence is not metrizable.
However, the space of probability measures on a complete, separable,
metric (Polish) space endowed with the topology of weak convergence is
itself a complete, separable, metric space \cite{Billingsley}. The
Prohorov metric, for example, can be used to metrize this space. This
metric has found many applications in information theory and
stochastic control. Furthermore, there are well-known conditions to
identify whether a family of probability measures is weakly compact
\cite{Billingsley}. For these reasons, one would like to work with
weak convergence. However, as we will observe, weak convergence is
insufficient in a general setup for obtaining continuity.

Before proceeding further, however, we look for conditions under which
an optimal control policy exists; i.e, when the infimum in $
\inf_{\gamma} E^{Q,\gamma}_{P} [c(X,U)]$ is a minimum. The following
simple result is proved in the Appendix.

\begin{theorem}\label{InfimumMinimum}
Suppose assumptions \textsc{A3} and \textsc{A4} hold. Then, there exists an
optimal control policy for any channel $Q$.
\end{theorem}

\textsc{Remark.}\   The assumptions that $c$ is bounded and
$\mathbb{U}$ is compact can be weakened in the preceding
theorem. For example, one can prove the same result by assuming that
$\mathbb{U} =\mathbb{R}^k$, $\lim_{\|u\|\to \infty} c(x,u)=\infty$ for
  all $x$, $c(x,u)$ is lower semi-continuous on $\mathbb{U}$ for every $x$, and there exists $u_0$ such that $\int
  c(x,u_0)P(dx)<\infty$.

\subsection{Weak convergence}

\subsubsection{Absence of continuity under weak convergence}

The following counterexample demonstrates that $J(P,Q)$ may not be
continuous under weak convergence of channels even for
continuous cost functions and compact $\mathbb{X}$, $\mathbb{Y}$, and
$\mathbb{U}$. Note that the absence of continuity here is also implied
by a less elementary counterexample for setwise convergence in
Section~\ref{sec_disc_setwise}.

Let $\mathbb{X}= \mathbb{Y}=\mathbb{U}=[a,b]$ for some
$a,b\in \mathbb{R}$, $a<b$. Suppose the cost is given as
$c(x,u)=(x-u)^2$
 and assume that $P$ is a discrete distribution with two atoms:
\[
P=\frac{1}{2}\delta_{a} + \frac{1}{2}\delta_{b},
\]
where $\delta_{a}$ is the delta measure at point $a$, that is,
$\delta_a(A)=1_{\{a \in A\}}$ for every Borel set $A$, where
$1_E$ denotes the indicator function of  event $E$.
Let $\{Q_n\}$ be a sequence of channels given by
\begin{equation}
Q_n(\,\cdot\,|x)=
\begin{cases} \delta_{a + \frac{1}{n}} & \text{if} \quad x \geq a+ \frac{1}{n},
\\
\delta_{a} & \text{if} \quad x < a+ \frac{1}{n}.
\end{cases}
\end{equation}

In this case, the optimal control policy, which is unique up to
changes in points of measure zero, is
\[
\gamma_n(y)= a1_{\{y < a +
  \frac{1}{n}\}} + b1_{\{y \geq a + \frac{1}{n}\}}, \quad n \in
\mathbb{N}, \quad n \geq \frac{1}{b-a},
\]
leading to a cost of $0$. We observe that the limit of the sequence
$\{Q_n(\,\cdot\,|x)\}$ is given by
\begin{equation}
Q(\,\cdot\,|x)= \delta_{a} \quad \text{for all $x \in \mathbb{R}$}.
\end{equation}
Thus, by Lemma~\ref{UniversalInputCompactness}, $Q_n\to Q$ weakly at
input $P$. However, the limit of the sequence of channels cannot distinguish between the inputs, since the channel output always equals $a$. Thus, even though
\[
J(P,Q_n) = 0, \quad   \text{for all $n \geq \dfrac{1}{b-a}$,}
\]
the cost of $Q=\lim_n Q_n$ is
\[
J(P,  Q) = \frac{(b-a)^2}{4}
\]
since, letting  $(X,Y)\sim PQ$, we have
$\gamma(y)=E[X|Y=y]=(b+a)/2$ for all $y$.  \qed

\subsubsection{Upper semi-continuity under weak convergence}
\mbox{}

\begin{theorem}\label{USCWeak}
Suppose assumptions \textsc{A1}  and \textsc{A5} hold.
If $\{Q_n\}$ is a sequence of channels converging weakly
at input $P$ to a channel $Q$, then
\[
  \limsup_{n\to \infty} J(P,Q_n) \le J(P,Q),
\]
that is, $J(P,Q)$ is upper semi-continuous on ${\cal Q}$ under
 weak convergence.
\end{theorem}

\begin{proof}  Let $\mu$ be an arbitrary probability measure on $(
\mathbb{X}\times \mathbb{Y}, \mathcal{B}(\mathbb{X}\times
\mathbb{Y}))$ and let $\mu_{\mathbb{Y}}$ be its  second marginal,
i.e., $\mu_{\mathbb{Y}}(A)= \mu(\mathbb{X}\times A)$ for  $A\in
\mathcal{B}(\mathbb{Y})$. Let $g\in \mathcal{G}$ be arbitrary.
By Lusin's theorem \cite[Thm.~2.24]{Rud87} there is a continuous
function\footnote{Lusin's theorem as stated in \cite{Rud87} implies
  the statement for $\mathbb{U}=\mathbb{R}$. The extension to the case
  $\mathbb{U}=\mathbb{R}^K$ is straightforward. If $\mathbb{U}$ is any
  closed and convex subset of $\mathbb{R}^K$, then there is a
  continuous function $\pi: \mathbb{R}^K \to \mathbb{U}$ such that
  $\pi(u)=u$ on $\mathbb{U}$ (the metric projection onto
  $\mathbb{U}$). Then $\hat{f}= \pi\circ f$ is the desired continuous
  mapping from $\mathbb{Y}$ into $\mathbb{U}$. } $f: \mathbb{Y}\to
\mathbb{U}$ such that
\[
  \mu_{\mathbb{Y}}\{y: \, f(y)\neq g(y)\} < \epsilon.
\]
Letting  $B=\{y: \, f(y)\neq g(y)\}$ we obtain
\begin{eqnarray*}
  \int\bigl| c(x,g(y)) -  c(x,f(y))\bigr| \mu(dx,dy)
  &= &  \int_{\mathbb{X}\times
    B}\bigl| c(x,g(y))-
  c(x,f(y))\bigr| \mu(dx,dy) \\
&< & \epsilon\cdot c^*,
\end{eqnarray*}
where $c^*= \sup_{x,u} c(x,u)<\infty$ by assumption \textsc{A1}, so
that
\begin{equation}
\label{eps_approx}
\int c(x,f(y)) \mu(dx,dy) <   \int c(x,g(y)) \mu(dx,dy) + c^*\epsilon.
\end{equation}
Let $\mathcal{C}$ be the set of continuous functions from 
$\mathbb{Y}$ into $\mathbb{U}$,  define
\[
j(\mu,\mathcal{C})= \inf_{\gamma \in \mathcal{C}}  \int c(x,\gamma(y))
\mu(dx,dy), \qquad j(\mu,\mathcal{G})= \inf_{\gamma\in \mathcal{G}}
\int c(x,\gamma(y)) 
\mu(dx,dy)
\]
and note that $j(\mu,\mathcal{C})\ge j(\mu,\mathcal{G})$ since
$\mathcal{C}\subset \mathcal{G}$.  By \eqref{eps_approx},
$j(\mu,\mathcal{C})$ is upper bounded by the right-hand-side of
\eqref{eps_approx}. Since $g$ in \eqref{eps_approx} was arbitrary, we
obtain $j(\mu,\mathcal{C})\le j(\mu,\mathcal{G})+c^*\epsilon $, which
in turn implies $j(\mu,\mathcal{C})\le j(\mu,\mathcal{G})$ since
$\epsilon>0$ was arbitrary. Hence $j(\mu,\mathcal{C})=
j(\mu,\mathcal{G})$.

Applying the above first to $PQ_n$ and then to $PQ$, we obtain
\begin{eqnarray*}
  \limsup_{n\to \infty}\inf_{\gamma\in \mathcal{G}} \int c(x,\gamma(y))
PQ_n(dx,dy)& = &   \limsup_{n\to \infty}\inf_{f\in \mathcal{C}}
\int c(x,f(y))
PQ_n(dx,dy) \\
&\le &    \inf_{f \in \mathcal{C}}  \limsup_{n\to \infty}
\int c(x,f(y)) PQ_n(dx,dy) \\
&=&   \inf_{f \in \mathcal{C}} \int c(x,f(y)) PQ(dx,dy) \\
&=&   \inf_{\gamma \in \mathcal{G}} \int c(x,\gamma(y)) PQ(dx,dy)
\end{eqnarray*}
where the next to last equality holds since $PQ_n$ converges weakly
to $PQ$.
\end{proof}

\subsection{Continuity properties under setwise convergence}

\subsubsection{Absence of continuity under setwise convergence}
\label{sec_disc_setwise}
The following counterexample demonstrates that $J(P,Q)$ may not be
continuous under setwise convergence of channels even for
continuous cost functions and compact $\mathbb{X}$, $\mathbb{Y}$, and
$\mathbb{U}$.

Let $\mathbb{X}=\mathbb{Y}=\mathbb{U}=[0,1]$. Assume that $X$ has
distribution
\[
P=\frac{1}{2}\delta_0 + \frac{1}{2} \delta_1.
\]
Let $Q(\,\cdot\,|x)= U([0,1])$ for all $x$,  so that if $(X,Y)\sim PQ$, then
$Y$ is independent of $X$ and has  the uniform
distribution on $[0,1]$.  Let $c(x,u)=(x-u)^2$.

By independence, $E[X|Y]=E[X]=1/2$, so
\begin{eqnarray*}
J(P,Q) &= &  \min_{\gamma\in   {\cal G}}E[(X-\gamma(Y))^2]=
E[(X-E[X|Y])^2] \\
 & =&  \frac{1}{2}\left(1-\frac{1}{2}\right)^2+
\frac{1}{2}\left(0-\frac{1}{2}\right)^2 = \frac{1}{4}.
\end{eqnarray*}

For $n\in \mathbb{N}$ and $k=1,\ldots,n$  consider the intervals
\begin{eqnarray}\label{star}
  L_{nk}= \left[\frac{2k-2}{2n},\frac{2k-1}{2n}\right), \quad
    R_{nk}= \left[\frac{2k-1}{2n},\frac{2k}{2n}\right)
\end{eqnarray}
and define  the ``square wave'' function
\[
  h_n(t) = \sum_{k=1}^n\bigl(1_{\{t \in L_{nk}\}} - 1_{\{t \in R_{nk}\}} \bigr).
\]
Since $\int_0^1 h_n(t)\, dt=0$ and $|h_n(t)|\leq 1$, the function
\[
  f_n(t)= \bigl(1+h_n(t) \bigr)1_{\{t \in [0,1]\}}
\]
is a probability density function.  Furthermore, the proof of the
Riemann-Lebesgue lemma (for example \cite{WhZy77}, Thm.\ 12.21) can be used
almost verbatim to show  that
\begin{eqnarray}
\lim_{n\to \infty}  \int_0^1 h_n(t) g(t)\, dt =0 \quad \text{for all
  $g\in L_1([0,1],\mathbb{R})$} \nonumber
\end{eqnarray}
and therefore
\begin{eqnarray}\label{Riemann}
\lim_{n\to \infty} \int_0^1 f_n(t) g(t)\, dt =\int_0^1 g(t)\, dt \quad
\text{for all $g\in L_1([0,1],\mathbb{R})$}.
\end{eqnarray}
In particular, we obtain that the sequence of  probability measures
induced by the sequence $\{f_n\}$ converges
setwise to $U([0,1])$.

Now, for every $n$, define a channel as
\[
Q_n(\,\cdot\,|x)=
\begin{cases}
  U([0,1]), & x=0\\
   \sim f_n, & x=1.
\end{cases}
\]
Then $Q_n(\cdot|x)\to Q$ setwise for $x=0$ and $x=1$, and thus $P
Q_n\to P U([0,1])$ setwise. However, letting $(X,Y_n)\sim P
Q_n$, a simple calculation shows that the optimal policy for $PQ_n$ is
\[
\gamma_n(y)= E[X|Y_n=y] =
  \begin{cases}
    0, &  y\in \bigcup_{k=1}^n R_{nk} \\
    \frac{2}{3},  &  y\in \bigcup_{k=1}^n L_{nk}
  \end{cases}
\]
and therefore for every $n \in \mathbb{N}$
\begin{eqnarray*}
  J(P,Q_n)&=&  \min_{\gamma\in
  {\cal G}}E[(X-\gamma(Y_n))^2]  \\
&= &  \frac{1}{2} \int_0^1 (0-\gamma_n(y))^2\, dy +
 \frac{1}{2} \int_0^1 (1-\gamma_n(y))^2f_n(y) \, dy \\
  &=& \frac{1}{6}.
\end{eqnarray*}
Thus, the optimal cost value is not continuous under setwise
convergence.  \qed

\smallskip

\subsubsection{Upper semi-continuity under setwise convergence}

\mbox{}

\begin{theorem}\label{thm11}
Under assumption \textsc{A2} the optimal cost
\[
J(P,Q) := \inf_{\gamma} E^{Q,\gamma}_{P} [c(X,U)]
\]
is sequentially upper  semi-continuous on  the set of communication
channels $\mathcal{Q}$  under setwise convergence.
\end{theorem}

\begin{proof} Let $\{Q_n\}$ converge setwise to $Q$ at input $P$. Then
\begin{eqnarray*}
  \limsup_{n\to \infty}\inf_{\gamma\in \mathcal{G}} \int c(x,\gamma(y))
PQ_n(dx,dy) &\le &    \inf_{\gamma\in \mathcal{G}}  \limsup_{n\to \infty}
\int c(x,\gamma(y)) PQ_n(dx,dy) \\
&=&   \inf_{\gamma\in \mathcal{G}}  \int c(x,\gamma(y)) PQ(dx,dy),
\end{eqnarray*}
where the equality holds since $c$ is bounded.
\end{proof}

\subsection{Continuity under total variation}

\mbox{}

\begin{theorem}\label{thm13}
Under assumption \textsc{A2}  the optimal cost
$J(P,Q)$ is is continuous on the set of communication channels ${\cal
  Q}$ under  under the topology of total variation.
\end{theorem}

\begin{proof}  Assume $Q_n\to Q$ in total variation at input $P$.  Let
$\epsilon>0$ and pick the $\epsilon$-optimal policies $\gamma_n$ and
$\gamma$ under channels $Q_n$ and $Q$, respectively. That is, letting
$\hat{J}(Q',\gamma')=E^{Q',\gamma'}_{P} [c(X,U)]$ for any $\gamma'\in
\mathcal{G}$ and $Q'\in \mathcal{Q}$, we have $ \hat{J}(Q_n,\gamma_n)
< J(P,Q_n) +\epsilon$ and $\hat{J}(Q,\gamma) < J(P,Q) +\epsilon$.

Considering first the case $J(P,Q_n)< J(P,Q)$,
we have
\begin{eqnarray*}
J(P,Q)-J(P,Q_n)&\le & J(P,Q) - \hat{J}(Q_n,\gamma_n)+\epsilon\\
& \leq & \hat{J}(Q,\gamma_n) - \hat{J}(Q_n,\gamma_n)+\epsilon.
\end{eqnarray*}
By a symmetric argument it
follows that
\begin{equation}\label{s33}
|J(P,Q)-J(P,Q_n)| \leq \max \bigl(\hat{J}(Q,\gamma_n) -
\hat{J}(Q_n,\gamma_n) , \hat{J}(Q_n,\gamma) - \hat{J}(Q,\gamma)
\bigr)+  \epsilon
\end{equation}
Now, since $c$ is bounded, it follows from (\ref{TValternative}) that
for any $\gamma'\in \mathcal{G}$,
\begin{eqnarray*}
|\hat{J}(Q_n,\gamma')-\hat{J}(Q,\gamma')| &=&
\bigg| \int    c(x,\gamma'(y)) PQ_n(dx,dy)  - \int
c(x,\gamma'(y))  PQ(dx,dy) \bigg| \\
& \leq & \|c\|_{\infty} \|PQ_n -
PQ\|_{TV}.
\end{eqnarray*}
This and (\ref{s33}) imply $|J(P,Q_n)- J(P,Q)|\le \|c\|_{\infty}
\|PQ_n - PQ\|_{TV} + \epsilon $. Since $\epsilon>0$ was arbitrary,
we obtain $|J(P,Q_n)- J(P,Q)|\le \|c\|_{\infty} \|PQ_n - PQ\|_{TV}$.
Since $ \|PQ_n - PQ\|_{TV}\to 0$ by
Lemma~\ref{UniversalInputCompactness}, we obtain $J(P,Q_n)\to J(P,Q)$
as claimed.
\end{proof}

\section{Problem {\bf P2}: Existence of optimal channels}\label{sectionExistence}

Here we study characterizations of  compactness which will be useful
in obtaining existence results. 

The discussion on weak convergence showed us that weak convergence
does not induce a strong enough topology, i.e., under which useful
continuity properties can be obtained. In the following, we will
obtain conditions for compactness for the other two convergence
notions, that is, for setwise convergence and total variation. We note
that in  the topologies induced by these three  modes of convergence, notions
of compactness and sequential compactness
coincide (for total variation and weak convergence this follows from
metrizability; for setwise convergence see
\cite[Thm.\ 4.7.25]{BogI07}).

We first discuss setwise convergence. A set of probability measures
$\mathcal{M}$ on some measurable space is said to be \emph{setwise
precompact} if every sequence in $\mathcal{M}$ has a
subsequence converging setwise to a probability measure (not
necessarily in $\mathcal{M}$). For two finite measures $\nu$ and $\mu$
defined on the same measurable space we write $\nu\le \mu$ if
$\nu(A)\le \mu(A)$ for all measurable $A$.

We have the the following condition for setwise (pre)compactness:

\begin{lemma}[{\cite[Thm.\ 4.7.25]{BogI07}}]  \label{Dunford}
Let $\mu$ be a finite measure on a measurable space
$(\mathbb{T},\mathcal{A})$. Assume a set of probability measures
$\Psi \subset {\cal 
  P}(\mathbb{T})$ satisfies
\[
P\le  \mu, \quad \text{for all $P\in \Psi$}.
\]
Then $\Psi$  is setwise precompact.
\end{lemma}

As before, $PQ \in {\cal P}(\mathbb{X} \times {\mathbb{Y}})$
denotes the joint probability measure induced by input $P$ and channel
$Q$, where $\mathbb{X}=\mathbb{R}^n$ and $\mathbb{Y}=\mathbb{R}^m$.  A simple consequence of the preceding majorization criterion is
the following.

\begin{lemma}\label{SetwiseCompact}
Let $\nu$ be a finite measure on ${\cal B}(\mathbb{X}\times
\mathbb{Y})$ and let $P$ be a probability measure on ${\cal
B}(\mathbb{X})$. Suppose $\mathcal{Q}$ is a set of channels
such that
\[
PQ \leq \nu, \quad \text{for all $Q\in \mathcal{Q}$}.
\]
Then $\mathcal{Q}$ is setwise  precompact at input $P$ in
the sense that any sequence in $\mathcal{Q}$ has a subsequence
$\{Q_n\}$ such that $Q_n\to Q$ setwise at input $P$ for some channel
$Q$.
\end{lemma}

\begin{proof}
By Lemma \ref{Dunford}, the set of joint measures
$\mathcal{M}=\{PQ:\, Q\in \mathcal{Q}\}$ is setwise 
precompact, that is, any sequence in $\mathcal{M}$ has a subsequence
$\{PQ_n\}$ converging to some $\hat{P}$ setwise.  Furthermore, since
the first marginal of $PQ$ is $P$ for all $n$, the first marginal of
$\hat{P}$ is also $P$ (since $PQ_n(A\times \mathbb{X})\to \hat{P}(A\times
\mathbb{X})$ for all $A\in \mathcal{B}(\mathbb{X})$). Now let $Q$ be
a  regular conditional probability measure satisfying
$\hat{P}=PQ$.
\end{proof}

For a probability density function $p$ on $\mathbb{R}^N$ we let $P_p$
denote the induced probability measure: $P_p(A)=\int_A p(x)\, dx$, $A\in
\mathcal{B}(\mathbb{R}^N)$. The next lemma gives a sufficient
condition for precompactness under total variation.

\begin{lemma}\label{TVcompactmeasure}
Let $\mu$ be a finite Borel measure on  $\mathbb{R}^N$ and
let $\mathcal{F}$ be an equicontinuous and uniformly bounded  family
of probability density functions.
Define  $\Psi \subset {\cal
P}(\mathbb{R}^N)$ by
\[
\Psi =\{P_p: P_p \leq \mu, \, p\in \mathcal{F}\}.
\]
Then  $\Psi$ is precompact under  total variation.
\end{lemma}

\begin{proof}
By Lemma~\ref{Dunford}, $\Psi$ is setwise precompact and
thus any sequence in $\Psi$ has a subsequence $\{P_n\}$ such that $P_n
\to P$ setwise for some $P\in \mathcal{P}(\mathbb{R}^N)$.
$P$ is clearly absolutely continuous with respect to the
Lebesgue measure on $\mathbb{R}^N$, and so it admits a density $p$.

Let $p_n$ be the density of $P_n$. It suffices to show that
\begin{equation}
\label{eql1conv}
\lim_{n\to \infty} \|p_n-p\|_1 = 0
\end{equation}
since $\|p_n-p\|_{TV}=2\|p_n-p\|_1= 2\int|p_n(x)-p(x)|\, dx$.

Pick a sequence of compact sets $K_j\subset \mathbb{R}^N$ such that
$K_j\subset K_{j+1}$ for all $j\in \mathbb{N}$, and
$\bigcup_{j}K_j=\mathbb{R}^N$.  Since the collection of densities
\{$p_n$\} is uniformly bounded and equicontinuous, it is precompact in
the supremum norm on each $K_j$ by the Arzel\`a-Ascoli theorem \cite{Dud02}. Thus
there exist subsequences $\{p_{n^j_k}\}$ such that
\[
\lim_{k\to \infty} \sup_{x\in K_j} |p_{n^j_k}(x) - p^j(x)| = 0
\]
for some continuous $p^j:\, K_j\to [0,\infty)$.

Since the $K_j$ are nested, one can choose $\{p_{n^{j+1}_k}\}$ to be a
subsequence of $\{p_{n^j_k}\}$ for all $j\in \mathbb{N}$. Then
$p^{j+1}$ coincides with $p^j$ on $K_j$ and we can define $\hat{p}$ on
$\mathbb{R}^N$ by setting $\hat{p}(x)=p^j(x)$, $x\in K_j$. We can now
use Cantor's diagonal method to pick an increasing sequence of
integers $\{m_i\}$ which is a subsequence of each $\{n^j_k\}$, and thus
\begin{equation}
\label{eqpwconv}
\lim_{i\to \infty} p_{m_i}(x) = \hat{p}(x), \quad  \text{for all $x \in
  \mathbb{R}^N$}.
\end{equation}
Note that by construction the convergence is uniform on each $K_j$
(and $\hat{p}$ is continuous). By uniform convergence
$P_{p_{m_i}}(A)\to P_{\hat{p}}(A)$ for all Borel subsets $A$ of
$K_j$. The setwise convergence of $P_n$ to $P_p$ implies
$P_{p_{m_i}}(A)\to P_{p}(A)$ for all Borel sets, so we must have
$p=\hat{p}$ almost everywhere. This and  (\ref{eqpwconv}) imply via
Scheff\'e's theorem  \cite{Bil86}
that
\[
\|p_{m_j}-p\|_1\to 0
\]
which completes the proof.
\end{proof}

The next result is an analogue of Lemma \ref{SetwiseCompact} and has
an essentially identical proof.
\begin{lemma}\label{TVCompact}
Let ${\cal Q}$ be a set of channels such that  $\{PQ: Q \in {\cal Q}\}$ is
a precompact set of probability measures under total variation.
Then $\mathcal{Q}$ is precompact under total variation at input $P$.
\end{lemma}

The following theorem, when combined with the preceding results, gives
sufficient conditions for the existence of best and worst channels
when the given family of channels $\mathcal{Q}$ is closed under the
appropriate convergence notion.

\begin{theorem}\label{ExistenceTheorem}
Recall problem \textsc{P2}.

\begin{romannum}

\item There exist a worst channel in $\mathcal{Q}$, that is, a
  solution for the maximization problem
\[
\sup_{Q \in {\cal Q}} J(P,Q) = \sup_{Q \in {\cal Q}} \inf_{\gamma}
E^{Q,\gamma}_{P} E[c(X,U)]
\]
when the set ${\cal Q}$ is weakly 
 compact and assumptions \textsc{A1}, \textsc{A4}, and \textsc{A5} hold.

\item There exist a worst channel in $\mathcal{Q}$ when the set ${\cal
  Q}$ is setwise  compact and assumption \textsc{A2} holds.

\item There exist best and worst channels in $\mathcal{Q}$,
  that is, solutions for the minimization problem $\inf_{Q \in {\cal
  Q}} J(P,Q)$ and the maximization problem $\sup_{Q \in {\cal Q}}
  J(P,Q)$ when the set ${\cal Q}$ is compact under total variation and
  assumption \textsc{A2} holds. \label{tvcase}
\end{romannum}
\end{theorem}

\begin{proof}
Under the stated conditions, we have  upper semi-continuity
or continuity (Theorems \ref{USCWeak}, \ref{thm11}, and \ref{thm13})
under the corresponding topologies. By compactness, the
existence of the cost maximizing (worst) channel follows when
 $J(P,Q)$ is upper-semicontinuous, while the existence of the
cost minimizing (best) channel follows when  $J(P,Q)$ is continuous in~$Q$.
\end{proof}

\textsc{Remark.}\
The existence of worst channels is useful for the robust control or
game-theoretic approach to optimization problems. If the problem is
formulated as a game where the uncertainty in the set is regarded as a
maximizer and the controller is the minimizer, one could search for a
max-min solution, which we prove to exist. One could also look for min-max
solutions, a topic which we leave as a future research topic. We note
that, in information theory, problems of similar nature have been
considered in the context of mutual information games
\cite{CoverThomasBook}.

\section{Application: quantizers as a class of channels}
\label{secapp}

Here we consider the problem of convergence and  optimization of
quantizers. We start with the definition of a quantizer.

\begin{definition} An $M$-cell {\em vector quantizer}, $q$, is a (Borel)
  measurable mapping from $\mathbb{X}=\mathbb{R}^n$ to the finite set
  $ \{1,2,\dots,M\}$, characterized by a measurable partition
  $\{B_1, B_2,\ldots,B_M\}$ such that $B_i= \{x: q(x)=i\}$ for $i
  =1,\ldots, M$. The $B_i$ are called the cells (or bins) of $q$.
\end{definition}

\textsc{Remarks.}
\begin{romannum}
\item For later convenience we allow for the possibility that
  some of the cells of the  quantizer  are empty.

\item Traditionally, in source coding theory, a quantizer is a mapping
  $q: \, \mathbb{R}^n \to \mathbb{R}$ with a finite range. Thus $q$ is
  defined by a partition and a reconstruction value in $\mathbb{R}^n$
  for each cell in the partition. That is, for given cells
  $\{B_1,\ldots,B_M\}$ and reconstruction values $\{
  c_1,\ldots,c_M\}\subset \mathbb{R}^n$, we have $q(x) = c_i$ if and
  only if $x \in
  B_i$. In our definition, we do not include the reconstruction
  values.
\end{romannum}

A quantizer $q$ with cells $\{B_1,\ldots,B_M\}$, however, can also be
characterized as a stochastic kernel $Q$ from  $\mathbb{X}$ to $\{1,\ldots,M\})$ defined by
\[
Q(i|x)= 1_{\{x \in B_i\}}, \quad  i=1,\ldots,M
\]
so that $ q(x) = \sum_{i=1}^M Q(i|x)$. We denote by ${\cal Q}_D(M)$
the space of all $M$-cell quantizers represented in the channel
form. In addition, we let ${\cal Q}(M)$ denote the set of (Borel) stochastic
kernels from  $\mathbb{X}$ to $\{1,\,\ldots,M\}$, i.e., $Q \in {\cal
  Q}(M)$ if and only if $Q(\,\cdot\,|x)$ is probability distribution
on $\{1,\ldots,M\}$ for all $x\in \mathbb{X}$, and $Q(i|\,\cdot\,)$ is
Borel measurable for all $i=1,\ldots,M$. Note that
$\mathcal{Q}_D(M)\subset \mathcal{Q}(M)$, and
by our definition $\mathcal{Q}_D(M-1)\subset
\mathcal{Q}_D(M)$ for all $M\ge 2$. We note that elements of
$\mathcal{Q}(M)$ are sometimes referred to in the literature as random
quantizers.

\begin{lemma}\label{osu}
The set of quantizers ${\cal Q}_D(M)$ is setwise 
precompact at any input $P$.
\end{lemma}

\begin{proof} Proof follows from  Lemma~\ref{SetwiseCompact} and the
  interpretation above regarding a quantizer as a channel. In
  particular, a majorizing finite measure $\nu$ is obtained by
  defining $\nu=P\times \lambda$, where $\lambda$ is the counting
  measure on $\{1,\ldots,M\}$ (note that
  $\nu(\mathbb{R}^n\times\{1,\ldots,M\}) =M$). Then for any measurable
  $B\subset \mathbb{R}^n$ and $i=1,\ldots,M$, we have
  $\nu(B\times\{i\})= P(B)\lambda(\{i\})=P(B)$ and so
\[
PQ(B\times \{i\}) = P(B\cap B_i) \le P(B) =  \nu   (B \times \{i\}).
\]
Since any measurable  $D \subset \mathbb{X}\times \{1,\ldots,M\}$ can
be written as the disjoint union of the sets $D_i\times \{i\}$,
$i=1,\ldots,M$, with $D_i=\{x\in
\mathcal{X}: (x,i)\in D\}$, the above implies $PQ(D) \le  \nu
(D)$.
\end{proof}

The following simple lemma provides a useful formula.

\begin{lemma}
\label{lem_quantconv}
A sequence $\{Q_n\}$ in $\mathcal{Q}(M)$  converges to a
  $Q$ in $\mathcal{Q}(M)$ setwise at input
  $P$ if and only if
\[
  \int_{A} Q_n(i|x)P(dx) \to   \int_{A} Q(i|x)P(dx) \quad
  \text{for all $A\in \mathcal{B}(\mathbb{X})$ and $i=1,\ldots,M$.}
\]
 \end{lemma}
\begin{proof} The lemma follows by noticing that for any $Q\in
  \mathcal{Q}(M)$ and  measurable
  $D \subset \mathbb{X}\times \{1,\ldots,M\}$,
\[
 PQ(D)=  \int_{D} Q(dy|x)P(dx) = \sum_{i=1}^M  \int_{D_i} Q(i|x)P(dx)
\]
where $D_i=\{x\in \mathcal{X}: (x,i)\in D\}$.
\end{proof}

The following counterexample shows that the space of
quantizers $\mathcal{Q}_D(M)$ is not closed under setwise convergence:

Let $\mathbb{X}=[0,1]$ and $P$ the uniform
distribution on $[0,1]$. Recall the definition $L_{nk}=
\left[\frac{2k-2}{2n},\frac{2k-1}{2n}\right)$ in (\ref{star}) and let
  $B_{n,1} = \bigcup_{i=1}^n L_{nk}$ and $B_{n,2}= [0,1]\setminus
  B_{n,1}$. Define $\{Q_n\}$ as the sequence of $2$-cell quantizers
  given by
\[
Q_n(1|x)=1_{\{x \in B_{n,1}\}}, \quad Q_n(2|x)=1_{\{x\in  B_{n,2}\}} .
\]
Then  (\ref{Riemann}) implies that for all $A\in \mathcal{B}([0,1])$,
\[
\lim_{n \to \infty} \int_{A}   Q_n(dy|x)P(dx) = \lim_{n\to \infty}
\int_0^1 \frac{1}{2}
f_n(t) \, dt = \frac{1}{2}P(A),
\]
and thus, by Lemma~\ref{lem_quantconv}, $Q_n$ converges setwise to
$Q$ given by $Q(1|x)=Q(2|x)=\frac{1}{2}$ for all $x\in [0,1]$.
However, $Q$ is not a (deterministic) quantizer.
\qed

\begin{definition} The class of \emph{finitely randomized quantizers}
  $\mathcal{Q}_{FR}(M)$ is the convex hull of
  $\mathcal{Q}_D(M)$, i.e., $Q\in \mathcal{Q}_{FR}(M)$
  if and only if there exist $k\in \mathbb{N}$,  $Q_1,\ldots,Q_k\in
  \mathcal{Q}_D(M)$, and $\alpha_1,\ldots,\alpha_k\in [0,1]$
  with $\sum_{i=1}^k \alpha_i=1$, such that
\[
Q(i | x) =  \sum_{j=1}^{k} \alpha_j Q_j(i|x), \quad \text{for all
  $i=1,\dots,M$ and $x\in \mathbb{X}$.}
\]
\end{definition}

The next result shows that $ {\cal Q}_{R}(M)$ is the
closure of the convex hull of $ {\cal Q}_D(M)$.

\begin{theorem}\label{BasisQuantizers}
For any $Q \in {\cal Q}(M)$ there exists a sequence $\{\hat{Q}_n\}$ of finitely randomized quantizers in ${\cal Q}_{FR}(M)$ which converges to $Q$ setwise at any input $P$.
\end{theorem}

\begin{proof} We will prove the existence of  a sequence
$\{ \hat{Q}_n \}$ in $ {\cal Q}_{FR}(M)$  such that $\
  \hat{Q}_n(\,\cdot\,|x)\to Q(\,\cdot\,|x)$ setwise  for all $x\in
  \mathbb{X}$.

Let ${\cal P}_M=\{z\in \mathbb{R}^M:\, z_1+\cdots + z_M=1,\, z_i\ge 0,\,
i=1,\ldots,M\}$ denote the probability simplex in $\mathbb{R}^M$ and note that
each $Q\in {\cal Q}({\cal M})$ is uniquely represented by the function $Q^v:\mathbb{X} \to
{\cal P}_M$ defined by
\[
  Q^v(x) = (Q(1|x),Q(2|x),\ldots,Q(M|x)).
\]
For a positive integer $n$ let ${\cal P}_{M,n}$ be the collection of
probability vectors in ${\cal P}_M$ with rational components having  common
denominator $n$, i.e.,
\[
{\cal P}_{M,n}= \bigl\{z\in {\cal P}_M: z_i
\in\{0,1/n,\ldots,(n-1)/n,1\},\, i=1,\ldots,M\bigr\}.
\]
Clearly,  any $z\in {\cal P}_M$ can be approximated
within error  $1/n$ in the $l_{\infty}$ sense by a member of ${\cal P}_{M,n}$,
i.e.,
\[
    \max_{z\in {\cal P}_M} \min_{z'\in {\cal P}_{M,n}} \|z-z'\|_{\infty}=    \max_{z\in {\cal P}_M} \min_{z'\in {\cal P}_{M,n}}
    \, \max_{i=1,\ldots,M}|z_i-z_i'| \le \frac{1}{n}.
\]
Breaking ties in a predetermined manner, we can make the selection of
$z'$ for a given $z$ unique, and thus define a Borel measurable
mapping $q_n:{\cal P}_M\to {\cal P}_{M,n}$ such that $z'=q_n(z)$
approximates $z$ in the above sense. Given $Q\in {\cal Q}({\cal M})$,
use this mapping to define $Q_n\in {\cal Q}({\cal M})$ through the
relation
\[
  Q_n^v(x) = q_n(Q^v(x)).
\]
(The measurability of $Q(i|x)$ in $x$ follows from the measurability
of the mapping $q_n$.)
Let $\{z^{(1)},\ldots,z^{(L(n))}\}$ be an enumeration of those  elements
of ${\cal P}_{M,n}$ for which the sets
\[
S_j= \{x: Q_n^v(x) = z^{(j)}\}, \quad j=1,\ldots,L(n)
\]
are not empty (clearly, $L(n)\le (n+1)^M$).  Note that the $S_i$  form
a Borel-measurable partition of $\mathbb{X}$ and we have
\[
  u:= (z^{(1)},z^{(2)},\ldots,z^{L(n)}) \in \bigl( {\cal P}_M
  \bigr)^{L(n)}
\]
and
\[
    Q_n^v(x)= z^{(j)} \quad \text{if $x\in S_j$}.
\]
Viewed as a subset of $\mathbb{R}^{M \cdot L(n)}$, the set $\bigl( {\cal P}_M
\bigr)^{L(n)}$ is compact and  convex and therefore by  the Krein-Milman
theorem (see, e.g., \cite{Bar02}) it is the closure of the convex hull
of its extreme points. The set of extreme points of $\bigl( {\cal P}_M
\bigr)^{L(n)}$ is $\bigl( {\cal E}_M \bigr)^{L(n)}$, where
${\cal E}_M=\{e_1,\ldots,e_M\}$ is the standard basis for $\mathbb{R}^M$.  In
particular, we can find
$u_1,\ldots,u_N \in \bigl( {\cal E}_M \bigr)^{L(n)}$ and
$(\alpha_1,\ldots,\alpha_N)\in {\cal P}_N$ such that $ \bigl\|u-\sum_{k=1}^N
\alpha_k u_k\bigr\|\le \frac{1}{n}$  ($\|\cdot\|$ denotes the standard
Euclidean norm in any dimension). Since
$u_k=(u_{k,1},\ldots,u_{k,L(n)})$, where $u_{k,j}\in {\cal E}_M$ for all $k$
and $j$, we can define the deterministic quantizers $Q_{n,k}\in {\cal Q}_D({\cal M})$,
$k=1,\ldots,N$,  by setting
\[
Q_{n,k}^v(x) = u_{k,j}  \quad \text{if $x\in S_j$}.
\]

Putting things together, we obtain that
\begin{equation}
\label{approx1}
       \biggl\|Q_n^v(x) -\sum_{k=1}^N \alpha_k Q_{n,k}^v(x)
      \biggr\| \le \frac{1}{n} \quad \text{for all $x\in \mathbb{X}$.}
\end{equation}
Define $\hat{Q}_n\in {\cal Q}(M)$ by
\[
\hat{Q}_n(i|x)= \sum_{k=1}^N \alpha_k Q_{n,k}(i|x).
\]
Combining (\ref{approx1}) with $\|Q^v(x) - Q_n^v(x)\|_{\infty} \le
\frac{1}{n} $, we obtain
\[
|Q(i|x)-\hat{Q}_n(i|x)|\le \frac{2}{n}\quad \text{for all $x\in \mathbb{X}$
  and $i=1,\ldots,M$}
\]
which implies that $\hat{Q}_n(\,\cdot\,|x)\to Q(\,\cdot\,|x)$ setwise for all
$x\in \mathbb{X}$.
Since each  $\hat{Q}_n$ is a convex combination of deterministic quantizers
in ${\cal Q}_D(M)$, the proof is complete.
\end{proof}

The preceding theorem has important consequences in that it tells us
that the space of deterministic quantizers is a ``basis'' for the
space of communication channels between $\mathbb{X}$ and $\{1,\ldots,M\}$
in an appropriate sense. In the following we  show that  an
optimal channel can be replaced with an optimal quantizer without any
loss in performance.

\begin{proposition}
For any $Q\in {\cal Q}(M)$ there is a $Q'\in
\mathcal{Q}_D(M)$ with $J(P,Q')\le J(P,Q)$. If there exists an
optimal channel in ${\cal Q}(M)$ for problem \textsc{P2},
then there is a quantizer in ${\cal Q}_D(M)$ that is optimal.
\end{proposition}

\begin{proof} Only the first statement needs to be proved.
We follow an argument common in the source coding literature (see,
e.g., the Appendix of \cite{Witsenhausen79}).

For a   policy $\gamma: \{1,\ldots,M\}\to
\mathbb{U}=\mathbb{X}$ (with finite cost) define for all $i$,
\[
B_i = \big\{ x:\,   c(x,\gamma(i)) \leq c(x,\gamma(j)) , \quad j=1,\ldots,M \big\}.
\]
Letting $B_1=\bar{B}_1$ and $B_i=\bar{B}_i\setminus
\bigcup_{j=1}^{i-1}B_j$, $i=2,\ldots,M$, we obtain a partition
$\{B_i,\ldots,B_M\}$ and a corresponding quantizer
$Q' \in {\cal Q}_D(M)$. It is easy to see that
$  E^{Q',\gamma}_{P} [c(X,U)]\le E^{Q,\gamma}_{P} [c(X,U)]$ for any
$Q\in {\cal Q}(M)$.
\end{proof}

The following shows that setwise convergence of quantizers implies
convergence under total variation.

\begin{theorem}\label{CompactTVandSetwise}
Let $\{Q_n\}$  be a sequence of quantizers  in ${\cal Q}_D(M)$ which converges to a quantizer $Q\in {\cal Q}_D(M)$
setwise at $P$. Then, the convergence is also under total variation at
$P$.
\end{theorem}

\begin{proof} Let $B^n_1,\ldots,B^n_M$ be the cells of $Q_n$.  Since
$Q_n \to Q$ setwise at input $P$, we have $PQ_n(B \times \{i\})\to
PQ(B \times \{i\})$ for any  $B \in {\cal B}(\mathbb{X})$.  Since
$PQ_n(B \times \{i\}) = \int_B 1_{\{x \in B^n_i \}} P(dx)$, we obtain
\[
P(B \cap B^n_i)  \to P(B \cap B_i), \quad \text{for all $i=1,\ldots,M$}.
\]
If  $B_1,\ldots,B_M$ are  the cells of $Q$, the above implies
$P(B_j \cap B^n_i)  \to P(B_j \cap B_i)$ for all $i,j\in \{1,\ldots,M\}$.
Since both $\{B^n_i\}$ and $\{B_n\}$ are partitions of $\mathbb{X}$,
we obtain
\[
  P(B^n_i\bigtriangleup B_i) \to 0 \quad \text{for all $i=1,\ldots,M$},
\]
where $B^n_i\bigtriangleup B=(B^n_i\setminus B)\cup (B\setminus
B^n_i)$. Then we have
\begin{eqnarray}
\lefteqn{\|PQ_n - PQ\|_{TV} } \nonumber \\
 &=&\sup_{f: \|f\|_{\infty} \leq 1}  \left|  \sum_{i=1}^M
\left( \int_{\mathbb{X}} f(x,i) Q_n(i|x)P(dx) -  \int_{\mathbb{X}}
f(x,i) Q(i|x)P(dx)\right)\right|  \nonumber \\
 &=&\sup_{f: \|f\|_{\infty} \leq 1}   \left| \sum_{i=1}^M
 \int_{\mathbb{X}} f(x,i)\bigl( 1_{\{x\in B^n_i\}}
-1_{\{x\in B^n_i\}}\bigr) P(dx)\right|  \nonumber \\
&\le &  \sup_{f: \|f\|_{\infty} \leq 1}  \sum_{i=1}^M
\int_{B^n_i \bigtriangleup B_i}  |f(x,i)| P(dx)
 \nonumber \\
&\le &  \sum_{i=1}^M P(B^n_i \bigtriangleup B_i ) \to
 0 \label{diff_conv}
\end{eqnarray}
and convergence in total variation follows.
\end{proof}

We next consider quantizers with convex codecells and an input
distribution that is absolutely continuous with respect to the
Lebesgue measure on $\mathbb{R}^n$ \cite{ConvexCodecell}.  Assume
$Q\in \mathcal{Q}_D(M)$ with cells $B_1,\ldots,B_M$, each of
which is a convex subset of $\mathbb{R}^n$. By the separating
hyperplane theorem,  there exist pairs
of complementary closed half spaces $\{(H_{i,j},H_{j,i}):\, 1\le
i,j\le M, i\neq j\}$ such that for all $i=1,\ldots,M$,
\[
  B_i \subset \bigcap_{j\ne i} H_{i,j}.
\]
Each $\bar{B}_i:= \bigcap_{j\ne i} H_{i,j}$ is a closed convex
polytope  and by the absolute continuity of $P$ one  has
$P(\bar{B}_i\setminus B_i)=0$ for all $i=1,\ldots,M$. One can thus
obtain a ($P$--a.s) representation of $Q$ by the $M(M-1)/2$
hyperplanes $h_{i,j}=H_{i,j}\cap H_{j,i}$.

Let $\mathcal{Q}_C(M)$ denote the collection of $M$-cell
quantizers with convex cells and consider a sequence $\{Q_n\}$ in
$\mathcal{Q}_C(M)$. It can be shown (see the proof of
Thm.~1 in \cite{ConvexCodecell}) that using an appropriate
parametrization of the separating hyperplanes, a subsequence $Q_{n_k}$
can be can be chosen which converges to a $Q \in
\mathcal{Q}_C(M)$ in the sense that $P(B^{n_k}_i
\bigtriangleup B_i )\to 0$ for all $i=1,\ldots,M$, where the
$B^{n_k}_i$ and the $B_i$ are the cells of $Q_{n_k}$ and $Q$,
respectively. In view of (\ref{diff_conv}), we obtain the following.

\begin{theorem}\label{compactConvexBins}
The set ${\cal Q}_C(M)$ is compact under total variation at
any input measure $P$ that is absolutely continuous with respect to
the Lebesgue measure on $\mathbb{R}^n$.
\end{theorem}

We can  now state an existence result for optimal quantization (problem
$\textsc{P1}$).

\begin{theorem}
\label{thmex}
Let $P$ be absolutely continuous and suppose the goal is to find the
best quantizer $Q$ with $M$ cells minimizing
$J(P,Q)=\inf_{\gamma}E_P^{Q,\gamma}(X,U)$ under assumption \textsc{A2},
where $Q$ is restricted to ${\cal Q}_C(M)$. Then an optimal
quantizer exists.
\end{theorem}

\begin{proof}
Existence follows from Theorems \ref{ExistenceTheorem} and
\ref{compactConvexBins}.
\end{proof}

In the quantization literature finding an optimal quantizer means
finding optimal codecells and corresponding reconstruction points. Our
formulation does not require the existence of optimal reconstruction
points (i.e., optimal policy $\gamma$).  For cost
functions of the form $c(x,u)=\|x-u\|^p$ for $x, u \in \mathbb{R}^n$
and some $p>0$, the cells of ``good'' quantizers will be convex by
Lloyd-Max conditions of optimality; see \cite{ConvexCodecell} for
further results on convexity of bins for entropy constrained
quantization problems. We note that \cite{AbayaWise} also considered
such cost functions for existence results on optimal quantizers; Graf
and Luschgy \cite{GrLu00} considered more general norm-based cost
functions.

\section{Multi-stage case}
\label{secmulti}
We consider the general case $T \in \mathbb{N}$.  It
should be observed that the effects of a control policy applied any
given time-stage presents itself in two ways, in both the cost
occurred at the given time-stage  and  the effect on the
process distribution at future time-stages, which is known as the
dual effect of control \cite{DualEffect}

The next theorem shows the continuity of the optimal cost in the
observation channel under some regularity conditions. Note that the
existence of best and worst channels follows under an appropriate
compactness condition as in Theorem~\ref{ExistenceTheorem} (iii). We
need the following definition.

\begin{definition} \label{ch_conv_unif}
A sequence of channels $\{Q_n\}$ converges to a channel $Q$
 \emph{uniformly} in total variation if
\[
\lim_{n\to \infty} \sup_{x\in \mathbb{X}}\, \bigl\|Q_n(\,\cdot\,|x) -
Q(\,\cdot\,|x) \bigr\|_{TV} =0.
\]
\end{definition}

Note that in the special but important case of additive observation
channels, uniform convergence in total variation is equivalent to the
weaker condition that $Q_n(\, \cdot\,|x) \to Q(\, \cdot\,|x)$ in total
variation for each $x$. When the additive noise is absolutely
continuous with respect to the Lebesgue measure, uniform convergence
in total variation is equivalent to requiring that the noise density
corresponding to $Q_n$ converges in the $L_1$ sense to the density
corresponding to $Q$. For example, if the noise density is
estimated from $n$ independent observations using any of the
$L_1$ consistent density estimates described in
e.g.\ \cite{DevroyeGyorfi}, then the resulting $Q_n$ will converge
(with probability one) uniformly in total variation.

\begin{theorem}\label{thm7}
Consider the cost function (\ref{Cost}) with arbitrary $T \in
\mathbb{N}$. Suppose assumption \textsc{A2} holds. Then, the
optimization problem \textsc{P1} is continuous in the observation
channel in the sense that if $\{Q_n\}$ is a sequence of channels
converging to $Q$ uniformly in total variation, then
\[
\lim_{n\to \infty} J(P,Q_n)= J(P,Q).
\]
\end{theorem}

\begin{proof} Let $\epsilon>0$ and pick  $\epsilon$-optimal policies
$\Pi^n=\{\gamma^n_0,\gamma^n_1,\dots,\gamma^n_{T-1}\}$ and
$\Pi=\{\gamma_0,\gamma_1,\dots,\gamma_{T-1}\}$ for channels $Q_n$ and
$Q$, respectively. That is, using the notation in (\ref{Cost}), we
have $ J(P,Q_n,\Pi^n) < J(P,Q_n) +\epsilon$ and $J(P,Q,\Pi) <
J(P,Q) +\epsilon$. The argument used to obtain (\ref{s33}) then gives
\begin{eqnarray}\label{s34}
\lefteqn{|J(P,Q)-J(P,Q_n)|} \nonumber \quad  \\
& \leq & \max \biggl(J(P,Q,\Pi^n) -
J(P,Q_n,\Pi^n) , J(P,Q_n,\Pi) - J(P,Q,\Pi)
\biggr)+  \epsilon.
\end{eqnarray}
We will show that  both terms in the maximum converge to zero.
First we consider the term
\begin{equation}
\label{eq_sumdiff}
J(P,Q^n,\Pi^n) - J(P,Q,\Pi^n) = \sum_{t=0}^{T-1} E_P^{Q^n,\Pi^n}
[c(X_t,U_t)] -  E_P^{Q,\Pi^n}
[c(X_t,U_t)].
\end{equation}
Under policy $\Pi^n=\{\gamma^n_0,\gamma^n_1,\ldots,\gamma^n_{T-1}\}$,
we have $U_t=\gamma^n_t(Y_{[0,t]},U_{[0,t-1]})$. We absorb in the
notation the dependence of $U_t$ on
$\gamma^n_0,\ldots,\gamma^n_{t-1}$ and write
$U_t=\gamma^n_t(Y_{[0,t]})$.

For $t=0,\ldots,T-1$ and  $k=0,\ldots,t$ define $\zeta^n_{k,t}:
\mathbb{X}^k \times \mathbb{Y}^k \to \mathbb{R}$ by  setting
\[
  \zeta^n_{t,t}(x_{[0,t]},y_{[0,t]}):= c(x_t, \gamma^n_t(y_{[0,t]})
\]
and defining recursively for $k=t-1,\ldots,0$
\[
  \zeta^n_{k,t}(x_{[0,k]},y_{[0,k]}) := \int
  P(dx_{k+1}|x_k,\gamma^n_k(y_{[0,k]}))Q_n(dy_{k+1}|x_{k+1})
  \zeta^n_{k+1,t} (x_{[0,k+1]},y_{[0,k+1]}).
\]
Note that $\|\zeta^n_{t,t}\|_{\infty}\le \|c\|_{\infty}$ and thus
$\|\zeta^n_{k,t}\|_{\infty}\le \|c\|_{\infty}$ for all
$k=t-1,\ldots,0$.

Fix $0\le k\le t$ and consider a system such that the observation
channel is $Q$ at stages $0,\ldots,k-1$ and $Q_n$ at stages
$k,k+1,\ldots,t$. Let $\mu^n_k$ denote the distribution of the
resulting process segment $(X_{[0,k]},Y_{[0,k]})$ under policy $\Pi^n$
(by definition $\mu^n_0=PQ_n$).  Also under policy $\Pi^n$, let
$\nu^n_k$ denote the distribution of $(X_{[0,k]},Y_{[0,k]})$ if the
observation channel is $Q$ for all the stages $0,\ldots,t$. Then we
have
\[
 E_P^{Q^n,\Pi^n}[c(X_t,U_t)] = \int \mu^n_0(dx_0,dy_0)\zeta^n_{0,t}(x_0,y_0)
\]
and
\[
 E_P^{Q,\Pi^n}[c(X_t,U_t)] = \int
 \nu^n_t(dx_{[0,t]},dy_{[0,t]})\zeta^n_{t,t}(x_{[0,t]},y_{[0,t]}).
\]
Note that by construction, for all $k=1,\ldots,t$
\begin{eqnarray*}
\lefteqn{\int \mu^n_k(dx_{[0,k]},dy_{[0,k]})
 \zeta^n_{k,t}(x_{[0,k]},y_{[0,k]}) } \qquad\qquad  \\
& = & \int \nu^n_{k-1}(dx_{[0,k-1]},dy_{[0,k-1]})
 \zeta^n_{k-1,t}(x_{[0,k-1]},y_{[0,k-1]}).
\end{eqnarray*}
Thus each term in the sum on the right hand side of
(\ref{eq_sumdiff}) can be expressed as a telescopic sum, which in turn
can be bounded term-by-term, as follows:
\[\quad \]
\begin{eqnarray}
\bigl|E_P^{Q^n,\Pi^n} [c(X_t,U_t)] -  E_P^{Q,\Pi^n}
 [c(X_t,U_t)]  \bigr|  &=& \biggl| \sum_{k=0}^t \int \mu^n_k(dx_{[0,k]},dy_{[0,k]})
\zeta^n_{k,t}(x_{[0,k]},y_{[0,k]})   \nonumber  \\
& & \mbox{} -  \int \nu^n_k(dx_{[0,k]},dy_{[0,k]})
\zeta^n_{k,t}(x_{[0,k]},y_{[0,k]}) \biggr|  \nonumber  \\
&\le & \sum_{k=1}^t \|\mu^n_k-\nu^n_k\|_{TV}
 \|\zeta^n_{k,t}\|_{\infty}   \nonumber \\
&\le & \|c\|_{\infty} \sum_{k=1}^t \|\mu^n_k-\nu^n_k\|_{TV}. \label{sumTVbound}
\end{eqnarray}

For any Borel set  $B\subset \mathbb{X}^k\times \mathbb{Y}^k$, define
$
B(x_{[0,k]},y_{[0,k-1]}) =\{y_k\in \mathbb{Y}:\, (x_{[0,k]},y_{[0,k]}) \in B\}
$, so that
\begin{eqnarray*}
 |\mu^n_k(B)-\nu^n_k(B)| &=& \biggl|\int
 \nu^n_{k-1}(dx_{[0,k-1]},dy_{[0,k-1]}) \int
   P(dx_k|x_{k-1},\gamma^n_{k-1}(y_{[0,k-1]}) \\
  & & \quad  \biggl(
   Q_n(B(x_{[0,k]},y_{[0,k-1]})|x_k)-
 Q(B(x_{[0,k]},y_{[0,k-1]})|x_k)\bigg)\biggr| \\
  &\le & \sup_{x_k\in \mathbb{X}} \|Q_n(\,\cdot\, |x_k)- Q_n(\,
 \cdot\, |x_k)\|_{TV}.
\end{eqnarray*}
The preceding bound and the uniform convergence of $\{Q_n\}$ imply $\lim_n
\|\mu^n_k-\nu^n_k\|_{TV}=0$ for all $k$.  Combining this with
(\ref{sumTVbound}) and (\ref{eq_sumdiff}) gives
\[
J(P,Q^n,\Pi^n) - J(P,Q,\Pi^n) \to 0.
\]
Replacing $\Pi^n$ with $\Pi$ we can use an identical argument to show
that $J(P,Q^n,\Pi) \to J(P,Q,\Pi)$. Since $\epsilon>0$ in (\ref{s34})
was arbitrary, the proof is complete.
\end{proof}

We obtained the continuity of the optimal cost on the space of
channels equipped with a more stringent notion for convergence in
total variation. This result and its proof indicate that
further technical complications emerge in multi-stage problems. Likewise,
upper semi-continuity under weak convergence and setwise convergence
require more stringent uniformity assumptions, which we leave for
future research.

One further interesting problem regarding the multi-stage case is to
consider adaptive observation channels. For example, one may aim to
design optimal adaptive quantizers for a control problem. In this
case, Markov Decision Process tools can be used for obtaining
existence conditions for optimal channels and quantizers. Some related
results on optimal adaptive quantization are presented in
\cite{BorkarMitterTatikonda}.

\section{Concluding remarks, some implications and future work}\label{sectionConclusion}
This paper studied the structural and topological properties of some
optimization problems in stochastic control in the space of
observation channels. The main problem we considered is how to
approach appropriate  notions of convergence and distance while studying
communication channels in the context of stochastic control problems.

The restriction to Euclidean state spaces is not essential and many
(but not all) of the positive results can be extended to the case
where $\mathbb{X}$, $\mathbb{Y}$, and $\mathbb{U}$ are arbitrary
Polish spaces. In particular, all the positive results in Sections
\ref{sectionCont} carry through without change, except
Theorem~\ref{USCWeak}. The results of Section~\ref{sectionExistence}
hold for this more general setup (however, in
Lemma~\ref{TVcompactmeasure} we need the additional condition that the
space is $\sigma$-compact). Likewise,  most of the positive results in
Section~\ref{secapp} on quantization hold more generally (in fact,
Theorem~\ref{BasisQuantizers} holds for an arbitrary measurable
space), but two of the main results, Theorems \ref{compactConvexBins}
and \ref{thmex}, do need the assumption that $\mathbb{X}$ is a
finite-dimensional Euclidean space.

\subsection{Sufficient conditions for continuity under setwise and
  weak convergence}

A careful analysis of the proof of Theorem \ref{thm13} reveals that we
need a uniform convergence principle for setwise convergence to be
sufficient for continuity.

That is, we wish to have
\begin{equation}\label{UniformConvergenceSetwise}
\lim_{n \to \infty}  \sup_{\gamma \in \mathcal{F}} \bigg| \int
\bigg( \int Q(dy|x) c(x,\gamma(y))  - \int Q_n(dy|x) c(x,\gamma(y))
\bigg) P(dx) \bigg| =0,
\end{equation}
where $\mathcal{F}$ is a set of allowable policies, to be able to have continuity under setwise convergence. Thus, one
important question of practical interest, is the following: What
type of stochastic control problems, cost functions, and allowable
policies lead to solutions which admit  such a uniform convergence  principle
under setwise convergence? Some sufficient conditions for uniform
setwise convergence are presented in \cite{Topsoe}.

Likewise, a parallel discussion applies for weak convergence under the
assumption that for every $Q_n$ and for $Q$, corresponding optimal
policies $\gamma_n$ and $\gamma$ are continuous and are assumed to be
from a given class of policies $\mathcal{F}$. One wants to have
\[
\int_{\mathbb{X}\times \mathbb{Y} } c(x,\gamma_n(y)) Q_n(dy|x)  P(dx)
 \to \int_{\mathbb{X}\times \mathbb{Y} }
c(x,\gamma(y))  Q(dy|x) P(dx).
\]

A sufficient condition for this is the following form of uniform weak
convergence:
\[
\lim_{n \to \infty} \sup_{\gamma \in {\cal F}}
\bigg|\int_{\mathbb{X}\times \mathbb{Y} }
c(x,\gamma(y)) Q_n(dy|x)P(dx)   - \int_{\mathbb{X}\times \mathbb{Y} }
c(x,\gamma(y))   Q(dy|x) P(dx)\bigg | = 0.
\]

\subsection{Empirical consistency of optimal controllers}

One issue to discuss is the connections of  our
results with {\em consistency} in learning the channel from empirical
observations.

When one does not know the system dynamics, such as the observation
channel, one typically attempts to learn the channel via test inputs
or empirical observations.  Let $\{(X_i,Y_i),\, i \in \mathbb{N} \}$
be an $\mathbb{X} \times \mathbb{Y}$-valued i.i.d sequence generated
according to some distribution $\mu$. Define the the empirical
occupation measures for every $n \in \mathbb{N}$, by letting
\[
\mu_n(B)=
\frac{1}{n}\sum_{i=1}^{n} 1_{\{(X_i,Y_i) \in B\}},
\]
for every measurable $B \subset \mathbb{X} \times \mathbb{Y}$. Then
one has $\mu_n(B) \to \mu(B)$ almost surely (a.s$.$) by the strong law
of large numbers. However, it is generally not true that $\mu_n\to
\mu$ setwise a.s.\ (e.g., $\mu_n$ never converges to $\mu$ setwise
when either $X_i$ or $Y_i$ has a nonatomic distribution), in which
case $\mu_n$ cannot converge to $\mu$ in total variation.

On the other hand, again by the strong law, for any  $\mu$-integrable
function $f$ on
$\mathbb{X} \times \mathbb{Y}$, one has, almost surely,
\[
\lim_{n \to \infty} \int f(x,y) \mu_n(dx, dy) = \int f(x,y) \mu(dx,
dy)
\]
In particular, $\mu_n \to \mu$ weakly with probability one
\cite{Dud02}.

In the learning theoretic context, the convergence of the costs
optimal for $\mu_n$ to the cost optimal for $\mu$ is called the
consistency of empirical risk minimization (see \cite{Vap00} for an
overview).  In particular, if the cost function and the allowable
control policies $\mathcal{F}$ are such that
\[
\lim_{n \to \infty} \sup_{\gamma \in {\cal F}} \bigg|\int
c(x,\gamma(y)) \mu_n(dx, dy) - \int c(x,\gamma(y)) \mu(dx, dy) \bigg|
=0,
\]
then we obtain consistency.

A class of measurable functions ${\cal E}$ is
called a {\em Glivenko-Cantelli class} \cite{Dudley}, if the integrals
with respect to the empirical measures converge almost surely to the
integrals with respect to the true measure uniformly over ${\cal
  E}$. Thus, if \[{\cal G}=\{\gamma: c(x,\gamma(y)) \in {\cal E}\},\]
where ${\cal E}$ is a class of Glivenko-Cantelli family of functions,
then we could establish consistency. One example of a
Glivenko-Cantelli family of real functions on $\mathbb{R}^N$ is the
family $\{f:\, \|f\|_{BL}\le M\}$ for some $0<M<\infty$, where
$\|\,\cdot\,\|_{BL}$ denotes the bounded Lipschitz norm  \cite{Dudley}.

Thus, if we restrict the class of control policies, and given a cost
function, we can obtain consistency and robustness to mismatch in the
channel due to learning. The classification of  the class of objective
functions and policies  which would lead to such a consistency result
is a future research problem.

\section{Appendix}
\subsection{Proof of Lemma \ref{UniversalInputCompactness}}
(i) \ Since $c(x,\,\cdot\,)$ is
continuous and bounded on $\mathbb{Y}$ for all $x$,  we have
\begin{eqnarray*}
  \lim_{n\to \infty} \int_{\mathbb{X}\times \mathbb{Y} } c(x,y) PQ_n(dx \, dy)
& =&
  \lim_{n\to \infty} \int_{\mathbb{X}}  \left( \int_{\mathbb{Y}}
  c(x,y) Q_n(dy|x) \right)  P(dx) \\
&=& \int_{\mathbb{X}}  \left( \int_{\mathbb{Y}}
  c(x,y) Q(dy|x) \right)  P(dx) \\
&=&
 \int_{\mathbb{X}\times \mathbb{Y} } c(x,y) PQ(dx,dy)
\end{eqnarray*}
where first we used  Fubini's theorem, and then the dominated convergence
theorem and the fact that $ \int_{\mathbb{X}}
  c(x,y) Q_n(dy|x)$ is bounded and converges to   $\int_{\mathbb{X}}
  c(x,y) Q(dy|x)$ for $P$-a.e.\ $x$.
\medskip

\noindent (ii)\ Let  $A \in {\cal B}(\mathbb{X} \times \mathbb{Y})$
and for $x$, let $A_x=\{y:\, (x,y)\in A\}$.  Similarly to the previous
proof,
\begin{eqnarray*}
  \lim_{n\to \infty}  PQ_n(A)
& =&
  \lim_{n\to \infty} \int_{\mathbb{X}} Q_n(A_x|x)  P(dx) \\
&=&  \int_{\mathbb{X}} Q(A_x|x)  P(dx) \\
&=& PQ(A)
\end{eqnarray*}
by the dominated convergence theorem since $\lim_{n\to \infty}
Q_n(A_x|x)= Q(A_x|x) $  for $P$-a.e.\ $x$.

\medskip

\noindent (iii)\ We have
\begin{eqnarray*}
\sup_{A  \in  {\cal B}(\mathbb{X} \times \mathbb{Y})}   |PQ_n(A)-PQ(A)|
&=& \sup_{A\in  {\cal B}(\mathbb{X} \times \mathbb{Y})}
\left|\int_{\mathbb{X}} Q_n(A_x|x)  P(dx) -
  \int_{\mathbb{X}} Q(A_x|x)  P(dx) \right|  \\
&\le & \sup_{A\in  {\cal B}(\mathbb{X} \times \mathbb{Y})}
  \int_{\mathbb{X}}\bigl| Q_n(A_x|x) -
  Q(A_x|x)\bigr|  P(dx)  \\
&\le &   \int_{\mathbb{X}} \; \;  \sup_{B\in  {\cal B}(\mathbb{Y})}\bigl|
  Q_n(B|x) -    Q(B|x)\bigr|  P(dx).
\end{eqnarray*}
Since $ \sup\limits_{B\in  {\cal B}(\mathbb{Y})}\bigl|
  Q_n(B|x) -    Q(B|x)\bigr| \to 0$  for $P$-a.e.\ $x$, an application
  of the dominated convergence theorem completes the proof. \qed

\subsection{Proof of Theorem \ref{InfimumMinimum}}
We have
\[
 J(P,Q)= \inf_{\gamma\in \mathcal{G}} \int_{\mathbb{X}\times
 \mathbb{Y}} c(x,\gamma(y)) Q(dy|y)P(dx).
\]
Let $(X,Y)\sim PQ$ and let $P(\,\cdot\,|y)$ be the (regular) conditional
distribution of $X$ given $Y=y$. If $(PQ)_{\mathbb{Y}}$ denotes the
  distribution of $Y$, then
\begin{eqnarray*}
J(P,Q)&= & \inf_{\gamma\in \mathcal{G}} \int_{\mathbb{Y}}
\int_{\mathbb{X}} c(x,\gamma(y))P(dx|y) (PQ)_{\mathbb{Y}}(dy)\\
  &=& \int_{\mathbb{Y}}  \biggl( \inf_{u\in \mathbb{U}}
\int_{\mathbb{X}} c(x,u)P(dx|y) \biggr) (PQ)_{\mathbb{Y}}(dy).
\end{eqnarray*}
where the validity of the second equality is explained below.

By assumption \textsc{A3}, $c$ is bounded and  $c(x,u_n)\to
c(x,u)$ if $u_n\to u$ for all $x$; thus by the dominated convergence
theorem
\[
\int_{\mathbb{X}} c(x,u_n)P(dx|y) \to \int_{\mathbb{X}} c(x,u)P(dx|y)
\]
proving that $g(u,y)= \int_{\mathbb{X}} c(x,u)P(dx|y)$ is continuous
in $u$ for each $y$. Since $\mathbb{U}$ is compact, there exists
$\gamma^*(y)\in \mathbb{U}$ such that $g(\gamma^*(y),y)= \inf_{u\in
  \mathbb{U}}g(u,y)$.  A standard argument shows
that $\gamma^*: \mathbb{Y}\to \mathbb{U} $ can be taken to be
measurable (see, e.g., Appendix~D of \cite{HernandezLermaMCP}) and we have
\[
 J(P,Q)= \int_{\mathbb{X}\times
 \mathbb{Y}} c(x,\gamma^*(y)) Q(dy|y)P(dx). \tag*{\qed}
\]

\end{document}